\documentclass[a4paper,11pt]{amsart}

\usepackage{amsfonts,amssymb,mathtools,amsmath}
\usepackage{textcomp, color}

\theoremstyle{plain}

\newtheorem*{thmA}{Theorem A}

\newtheorem{thm}{Theorem}[section]
\newtheorem{lem}[thm]{Lemma}

\theoremstyle{definition}

\newtheorem{dfn}[thm]{Definition}

\newcommand{\F}{\mathbb{F}}
\newcommand{\Z}{\mathbb{Z}}

\newcommand{\N}{\mathbb{N}}

\newcommand{\C}{\mathbb{C}}
\renewcommand{\P}{\mathbb{P}}

\newcommand{\No}{\mathcal{N}}

\DeclareMathOperator{\Aut}{Aut}

\DeclarePairedDelimiter\ceil{\lceil}{\rceil}
\DeclarePairedDelimiter\floor{\lfloor}{\rfloor}

\begin{document}

\title{A note on strongly real Beauville $p$-groups}

\author[\c{S}.\ G\"ul]{\c{S}\"ukran G\"ul}
\address{Department of Mathematics\\ Middle East Technical University\\
	06800 Ankara, Turkey}
\email{gsukran@metu.edu.tr}

\keywords{Strongly real Beauville groups; free product \vspace{3pt}}

\thanks{ The author is supported by the Spanish Government, grant MTM2014-53810-C2-2-P, the Basque Government, grant IT974-16, and the ERC Grant PCG-336983.}

\begin{abstract}
	We give an infinite family of non-abelian strongly real Beauville $p$-groups for any odd prime $p$ by considering the lower central quotients of the free product of two cyclic groups of order $p$. This is the first known infinite family of non-abelian strongly real Beauville $p$-groups.
\end{abstract}	
	
\maketitle
	
\section{Introduction}

A \emph{Beauville surface\/} of unmixed type is a  compact complex surface isomorphic to 
$(C_1\times C_2)/G$, where $C_1$ and $C_2$ are algebraic curves of genus at least $2$ and $G$ is a finite group acting freely on $C_1\times C_2$ and faithfully on the factors $C_i$ such that $C_i/G\cong \P_1(\C)$ and the covering map $C_i\rightarrow C_i/G$ is ramified over three points for $i=1,2$.
Then the group $G$ is said to be a \emph{Beauville group\/}.

The condition for a finite group $G$ to be a Beauville group can be formulated in purely group-theoretical terms.

\begin{dfn}
For a couple of elements $x,y \in G$, we define
\[
\Sigma(x,y)
=
\bigcup_{g\in G} \,
\Big( \langle x \rangle^g \cup \langle y \rangle^g \cup \langle xy \rangle^g \Big),
\]
that is, the union of all subgroups of $G$ which are conjugate to $\langle x \rangle$, to 
$\langle y \rangle$ or to $\langle xy \rangle$. Then $G$ is a Beauville group if and only if the following conditions hold:
\begin{enumerate}
	\item $G$ is a $2$-generator group.
	\item There exists a pair of generating sets $\{x_1,y_1\}$ and $\{x_2,y_2\}$ of $G$ such that 
	$\Sigma(x_1,y_1) \cap \Sigma(x_2,y_2)=1$.
\end{enumerate}
Then $ \{x_1,y_1\}$ and $\{x_2,y_2\}$ are said to form a \emph{Beauville structure\/} for $G$. 
\end{dfn}

\begin{dfn}
	Let $G$ be a Beauville group. We say that $G$ is \emph{strongly real\/} if there exists a Beauville structure $\{\{x_1,y_1\}, \{x_2,y_2\} \}$ such that there exist an automorphism $\theta \in \Aut(G)$ and elements $g_i \in G$ for $i=1,2$ such that 
	\[
	g_i \theta(x_i)g_i^{-1}=x_i^{-1} \ \ \text{and} \ \ g_i \theta(y_i)g_i^{-1}=y_i^{-1}
	\]
	for $i=1,2$. Then the Beauville structure is called \emph{strongly real Beauville structure}.
\end{dfn}
In practice, it is convenient to take $g_1=g_2=1$. 

In 2000, Catanese \cite{cat} proved that a finite abelian group is a Beauville
group if and only if it is isomorphic to $C_n \times C_n$, where $n>1$ and $\gcd(n,6)=1$.  Since for any abelian group the function $x \longmapsto -x$ is an automorphism,
the following result is immediate.
\begin{lem}
	\label{abelian}
	Every abelian Beauville group is a strongly real Beauville group.
\end{lem}

Thus, there are  infinitely many abelian strongly real Beauville $p$-groups for $p\geq5$.
 
Recall that the only known infinite family of Beauville $2$-groups was constructed in \cite{BBPV}. However, one of the main results in \cite{BBPV} shows that these Beauville $2$-groups are not strongly real. On the other hand, in \cite{fai}, Fairbairn has recently given the following examples of strongly real Beauville $2$-groups.
The groups
\[
G= \langle x,y \mid x^8=y^8=[x^2,y^2]=(x^iy^j)^4=1 \ \text{for} \ i,j=1,2,3 \rangle,
\]
and
\[
G= \langle x,y \mid (x^iy^j)^4=1 \ \text{for} \ i,j=0,1,2,3 \rangle
\]
are strongly real Beauville groups of order $2^{13}$ and $2^{14}$, respectively.

If $p\geq3$ there is no known example of a non-abelian strongly real Beauville $p$-group. Thus, up to now the only examples of strongly real Beauville $p$-groups are the abelian ones and the two groups given above.

In this paper, we give infinitely many non-abelian strongly real Beauville $p$-groups for any odd prime $p$. To this purpose, we work with the lower central quotients of the free product of two cyclic groups of order $p$. The main result of this paper is as follows.

\begin{thmA}
	Let $F=\langle x,y \mid x^p , y^p \rangle$ be the free product of two cyclic groups of order $p$ for an odd prime $p$, and let $i= k(p-1)+1$ for $k\geq 1$. Then the quotient $F/\gamma_{i+1}(F)$ is a strongly real Beauville group.
\end{thmA}

 Note that in \cite{gul}, it was recently shown that all $p$-central quotients of the free product $F=\langle x,y \mid x^p , y^p \rangle$  are Beauville groups. Observe that since $F/F'$ has exponent $p$, the lower central series and $p$-central series of $F$ coincide.

\section{ Proof of the main theorem}

In this section, we give the proof of Theorem A. Let $F=\langle x,y \mid x^p,y^p \rangle$ be the free product of two cyclic groups of order $p$. We begin by stating a lemma regarding the existence of an automorphism of $F$ which sends the generators to their inverses. The proof is left to the reader.

\begin{lem}
	\label{isomorphism of F}
	Let $F=\langle x,y \mid x^p,y^p \rangle$ be the free product of two cyclic groups of order $p$. Then the map
	\begin{align*} 
	\theta \colon F & \longrightarrow F \\
	x & \longmapsto x^{-1}\\
	y & \longmapsto y^{-1},
	\end{align*}
	is an automorphism of $F$.
\end{lem}


Before we proceed, we will introduce some results regarding the Nottingham group which will help us to determine some properties of $F$.

The \emph{Nottingham group\/} $\No$ over the field $\F_p$, for odd $p$, is the (topological) group of normalized automorphisms of the ring $\F_p[[t]]$ of formal power series. For any positive integer $k$, the automorphisms $f\in\No$ such that $f(t)=t+\sum_{i\ge k+1} \, a_it^i$ form an open normal subgroup $\No_k$ of $\No$ of index $p^{k-1}$.
Observe that $|\No_k:\No_{k+1}|=p$ for all $k\ge 1$.
We have the commutator formula
\begin{equation}
\label{comms Ni}
[\No_k,\No_{\ell}]
=
\begin{cases}
\No_{k+\ell},
&
\text{if $k\not\equiv \ell\pmod p$,}
\\
\No_{k+\ell+1},
&
\text{if $k\equiv \ell\pmod p$}
\end{cases}
\end{equation}
(see \cite{cam}, Theorem 2).
Thus the lower central series of $\No$ is given by
\begin{equation}
\label{lcs}
\gamma_i(\No)= \No_{r(i)},
\quad
\text{where}
\quad
r(i)=i+1+\floor*{\frac{i-2}{p-1}}.
\end{equation}
As a consequence, $|\gamma_i(\No):\gamma_{i+1}(\No)|\le p^2$, and we have `diamonds' of order $p^2$ if and only if $i$ is of the form $i=k(p-1)+1$ for some $k\ge 0$.
In other words, the diamonds in the lower central series of $\No$ correspond to quotients $\No_{kp+1}/\No_{kp+3}$.

Recall that by Remark 3 in \cite{cam}, $\No$ is topologically generated by the elements $a\in \No_1 \smallsetminus \No_2$ and $b\in \No_2\smallsetminus \No_3$ given by
$a(t) = t(1-t)^{-1}$ and $b(t)= t(1-2t)^{-1/2}$, which are both of order $p$.

In the following lemma, we need a result of Klopsch \cite[formula (3.4)]{klo} regarding the centralizers of elements of order $p$ of $\No$ in some quotients
$\No/\No_k$.
More specifically, if $f\in \No_k\smallsetminus \No_{k+1}$ is of order $p$, then for every
$\ell=k+1+pn$ with $n\in\N$, we have
\begin{equation}
\label{centralizer}
C_{\No/\No_{\ell}}(f\No_{\ell}) = C_{\No}(f)\No_{\ell-k}/\No_{\ell}.
\end{equation}

\begin{lem}
	\label{non-covering}
	Put $G=\No/\No_{kp+3}$ and $N_i= \No_i/ \No_{kp+3}$ for $1\leq i \leq kp+3$.
	If $\alpha$ is the image of $a$ in $G$, then the set
	$\{[\alpha,g]\mid g\in G\}$ does not cover $N_{kp+1}$.
\end{lem}

\begin{proof}
	To prove the lemma, we  will show that $\{[\alpha,g] \mid g\in G \}\cap N_{kp+2}=1$. Assume that $[\alpha,g]\in N_{kp+2}$ for some $g\in G$. Since $a\in \No_1\smallsetminus \No_2$ is of order $p$, it follows from (\ref{centralizer}) that
	\[
	C_{\No/ \No_{kp+2}}(a\No_{kp+2})=C_{\No}(a)\No_{kp+1}/\No_{kp+2}.
	\]
	Thus we can write $g=ch$, with $[\alpha,c]=1$ and $h\in N_{kp+1}$.
	Then $[\alpha,g]=[\alpha,h]\in [G, N_{kp+1}]=1$, since
	$N_{kp+1}$ is central in $G$.
\end{proof}

\begin{lem}
	\label{non-covering in free product}
	Put $H=F/\gamma_{i+1}(F)$, where $i=k(p-1)+1$ for $k\geq 1$ and $H_i= \gamma_i(F)/ \gamma_{i+1}(F)$. If  $u$ and $v$ are the images of $x$ and $y$ in $H$, respectively, then the sets $\{[u, h] \mid h\in H\}$ and $\{[v, h] \mid h\in H\}$ do not cover $H_i$.
\end{lem}

\begin{proof}
Let $G= \No/\No_{kp+3}$, and let us call $\alpha$ and $\beta$ the images of $a$ and $b$ in $G$, respectively. Since $\alpha$ and $\beta$ are of order $p$ and $\gamma_{i+1}(G)=1$, the map
\begin{align*} 
\psi \colon H & \longrightarrow G \\
 u & \longmapsto \alpha \\
 v & \longmapsto \beta,
\end{align*}
is well-defined and an epimorphism.

By Lemma \ref{non-covering}, the set of commutators of $\alpha$ does not cover the subgroup $\gamma_i(G)=N_{kp+1}$.
 It then follows that the set $\{[u, h] \mid h\in H\}$ does not cover $H_i$.
Since the roles of $u$ and $v$ are symmetric, we also conclude that the set $\{[v, h] \mid h\in H\}$ does not cover $H_i$, as desired.

\end{proof}

To prove the main result, we need  the following three lemmas.

\begin{lem}
	\label{intersection1}
	Let $G= \langle  a, b\rangle$ be a $2$-generator  $p$-group and $o(a)=p$, for some prime $p$. Then
	\[
	\Big(\bigcup_{g\in G} {\langle a\rangle}^g  \Big)
	\bigcap
	\Big(\bigcup_{g\in G} {\langle b\rangle}^g \Big) 
	= 1.
	\]
\end{lem}

\begin{proof}
		We assume that $x= (a^i)^g=(b^j)^h$ for some $i,j \in \mathbb{Z}$ and $g, h \in G$, and prove that $x=1$. In the quotient 
		$\overline{G}= G/ \Phi(G)=\langle \overline{a} \rangle \times \langle \overline {b} \rangle$, we have $\overline {x} \in \langle \overline {a} \rangle \cap \langle \overline {b} \rangle = \overline {1}$ implying that  $x \in \Phi(G)$. On the other hand, $x \in \langle a^g \rangle$,
		where $a^g $ is of order $p$ and $a^g \notin \Phi(G)$. It then follows that $x=1$.
\end{proof}

\begin{lem} \textup{\cite[Lemma~3.8]{FG} }
	\label{intersection2}
	Let $G$ be a finite $p$-group and  let $x \in G \smallsetminus \Phi(G)$ be an element of order $p$. If $t \in \Phi(G)\smallsetminus\{[x,g]\mid g\in G\}$ then 
	\[
	\Big(\bigcup_{g\in G} {\langle x\rangle}^g  \Big)
	\bigcap
	\Big(\bigcup_{g\in G} {\langle xt\rangle}^g \Big)
	= 1.
	\]
\end{lem}

\begin{lem}\textup{\cite[Lemma~3.1]{gul} }
	\label{homomorphism}
	Let $\psi \colon  G_1 \to G_2$ be a group homomorphism, let $x_1, y_1 \in G_1$ and $x_2= \psi(x_1)$, $y_2=\psi(y_1)$. If $o(x_1)=o(x_2)$ then the condition $\langle x_2^{\psi(g)} \rangle \cap \langle y_2^{\psi(h)} \rangle= 1$ implies that $\langle x_1^{g} \rangle \cap \langle y_1^{h} \rangle= 1$ for  $g,h \in G_1$.
\end{lem}

 Let $H=F/\gamma_{i+1}(F)$ and let $u$ and $v$ be the images of $x$ and $y$ in $H$, respectively. In order to prove the main theorem, we need to know the order of $uv$. We first recall a result of Easterfield \cite{eas} regarding the exponent of $\Omega_j(G)$. More precisely, if $G$ is a $p$-group, then for every $j,k\geq 1$, the condition $\gamma_{k(p-1)+1}(G)=1$ implies that
 \begin{equation}
 \label{exp of omega}
 \exp \Omega_j(G)\leq p^{j+k-1}.
 \end{equation}
 
 If we set $k=\ceil*{\frac{i}{p-1}}$, we have $\gamma_{k(p-1)+1}(H)\leq \gamma_{i+1}(H)=1$. Then  by (\ref{exp of omega}), we get $\exp H \leq p^k$, and hence $o(uv)\leq p^k$. Indeed, we will show that $o(uv)=p^k$. To this purpose, we also need to introduce a result regarding $p$-groups of maximal class with some specific properties.

Let $G=\langle s \rangle \ltimes A$ where $s$ is of order $p$ and $A\cong \Z_p^{p-1}$. The action of $s$ on $A$ is via $\theta$, where $\theta$ is defined by the companion matrix of the $p$th cyclotomic polynomial $x^{p-1}+\dots +x+1$. Then $G$ is the only infinite pro-$p$ group of maximal class. Since $s^p=1$ and $\theta^{p-1}+\dots +\theta+1$ annihilates $A$, this implies that for every $a\in A$,
\[
(sa)^p
=
s^p a^{s^{p-1}+\dots +s+1}
=
1.
\]
Thus all elements in $G\smallsetminus A$ are of order $p$.

Let $P$ be a finite quotient of $G$ of order $p^{i+1}$ for $i\geq 2$. Let us call $P_1$ the abelian maximal subgroup of $P$ and $P_j=[P_1, P, \overset{j-1}{\ldots}, P]=\gamma_j(P)$ for $j\geq 2$. Then one can easily check that 
$\exp P_j= p^ {\ceil*{\frac{i+1-j}{p-1}} }$ and every element in $P_j\smallsetminus P_{j+1}$ is of order $p^ {\ceil*{\frac{i+1-j}{p-1}} }$.

Let $s\in P\smallsetminus P_1$ and $s_1\in P_1\smallsetminus P'$.  Since all elements in $P \smallsetminus P_1$ are of order $p$ and $\gamma_{i+1}(P)=1$, the map
\begin{align*} 
\psi \colon H & \longrightarrow P \\ 
u& \longmapsto s^{-1}\\
v& \longmapsto ss_1,
\end{align*}
is well-defined and an epimorphism. Then we have $o(uv) \geq o(s_1)=p^k$, and this, together with  $\exp H =p^k$, implies that $o(uv)=p^k$.

We are now ready to give the proof of main theorem.

\begin{thm}
	\label{strongly real}
	Let  $p\geq 3$, and let $i= k(p-1)+1$ for $k\geq1$. Then the quotient $F/\gamma_{i+1}(F)$ is a strongly real Beauville group.
\end{thm}

\begin{proof}
	Let $H$ and $H_i$ be as defined in Lemma \ref{non-covering in free product}. Let $u$ and $v$ be the images of $x$ and $y$ in $H$, respectively.
	By Lemma \ref{non-covering in free product}, there exist $w,z\in H_{i}$ such that
	$w\not\in \{[u,h]\mid h\in H\}$ and $z\not\in \{[v,h]\mid h\in H\}$. Observe that $w$ and $z$ are central elements of order $p$ in $H$. We claim that $\{u,v\}$ and 
	$\{(uw)^{-1},vz\}$ form a Beauville structure in $H$.
	Let $X=\{u,v,uv\}$ and $Y=\{(uw)^{-1},vz,u^{-1}vw^{-1}z\}$.
	
	Assume first that $x\in X$ is of order $p$, and let $y\in Y$.
	If $\langle x\Phi(H) \rangle \ne \langle y\Phi(H) \rangle$ in $H/\Phi(H)$, then by  Lemma 
	\ref{intersection1}, $\langle x\rangle^g \cap \langle y \rangle^h=1$ for every $g,h\in H$. Otherwise, we are in one of the following two cases: $x=u$ and $y=(uw)^{-1}$, or
	$x=v$ and $y=vz$. Then the condition $\langle x \rangle^g \cap \langle y \rangle^h=1$ follows by Lemma \ref{intersection2}.
	
	We now assume that  $x=uv$. Again applying  Lemma \ref{intersection1}, we get $\langle x\rangle^g \cap \langle y \rangle^h=1$ where $y= (uw)^{-1}$ or $y=vz$, which is of order $p$. Thus we are only left with the case when $x=uv$ and $y=u^{-1}vw^{-1}z$. Recall that the map $\psi\colon H \longrightarrow P $ is an epimorphism such that $\psi(u)=s^{-1}$ and $\psi(v)=ss_1$. Then $\psi(u^{-1}vw^{-1}z)$ is an element outside $P_1$, which is of order $p$ . Thus $\langle \psi(u^{-1}vw^{-1}z) ^{\psi(g)} \rangle \cap \langle s_1^{\psi(h)}\rangle =1$ for all $g,h \in H$. Since $o(uv)=o(s_1)$, the condition
	$\langle x \rangle^g \cap \langle y \rangle^h=1$ for all $g,h\in H$ follows by Lemma \ref{homomorphism}.
	This completes the proof that $G$ is a Beauville group.
	
	We next show that the Beauville structure $\{\{u,v\}, \{(uw)^{-1},vz\}\}$ is strongly real. By Lemma \ref{isomorphism of F}, we know that the map $\theta$ is an automorphism of $F$. Since $\theta(\gamma_n(F))= \gamma_n(\theta(F))=\gamma_n(F)$ for all $n\geq 1$, the map $\theta$ induces an automorphism $\overline{\theta} \colon H \longrightarrow H$ such that $\overline{\theta}(u)=u^{-1} $ and $\overline{\theta}(v)=v^{-1} $. 
	Now we only need to check if  $\overline{\theta} ((uw)^{-1})= uw$ and  $\overline{\theta} (vz)= (vz)^{-1}$. Note that 
	\[
	\overline{\theta} ((uw)^{-1})
	=\overline{\theta} (w^{-1})u
	=u\overline{\theta} (w^{-1}),
	\]
	and 
	\[
	\overline{\theta} (vz)= v^{-1}\overline{\theta}(z)=\overline{\theta}(z)v^{-1} 
	\]
	where the last equalities follow from the fact that both $w$ and $z$ are central in $H$. Thus it suffices to see that $\overline{\theta}(w^{-1})=w$ and $\overline{\theta}(z)=z^{-1}$.
	
	Note that $H_i$ is generated by the commutators of length $i$ in $u$ and $v$. Since $i$ is odd and $H_i \leq Z(G)$, it follows that
    \[
    \overline{\theta}([x_{j_1}, x_{j_2}, \dots, x_{j_i}])=[x_{j_1}^{-1}, x_{j_2}^{-1}, \dots, x_{j_i}^{-1}]=[x_{j_1}, x_{j_2}, \dots, x_{j_i}]^{-1},
    \]
    where each $x_{j_k}$ is either $u$ or $v$.	
	Hence the automorphism $\overline{\theta}$ sends the generators of $H_i$ to their inverses. Since $H_i$ is abelian, this implies that for every $t\in H_i$ we have $\overline{\theta}(t)=t^{-1}$. 
\end{proof}

\end{document}